\documentclass{article}
\usepackage{amsmath, amssymb}
\usepackage{amssymb,pb-diagram,pst-all,graphicx}
\usepackage{amscd}

\usepackage{fancybox}

\newcommand{\ZZ}{\mathbb Z}
\newcommand{\PP}{\mathbb P}
\newcommand{\QQ}{\mathbb Q}
\newcommand{\CC}{\mathbb C}

\newcommand{\mcB}{\mathcal B}
\newcommand{\mcC}{\mathcal C}

\newcommand{\mcQ}{\mathcal Q}

\newcommand{\bAlex}{\mathbf{Alex}}
\newcommand{\bCov}{\mathop {\rm \underline{Cov}}\nolimits}

\newcommand{\bde}{\boldsymbol {e}}

\newcommand{\Cov}{\mathop {\rm Cov}\nolimits}

\newcommand{\MW}{\mathop {\rm MW}\nolimits}
\newcommand{\Gal}{\mathop {\rm Gal}\nolimits}

\newcommand{\NS}{\mathop {\rm NS}\nolimits}

\newcommand{\rank}{\mathop {\rm rank}\nolimits}
\newcommand{\Red}{\mathop {\rm Red}\nolimits}

\newcommand{\Supp}{\mathop {\rm Supp}\nolimits}
\newcommand{\Sing}{\mathop {\rm Sing}\nolimits}

\newtheorem{thm}{Theorem}[section]
\newtheorem{cor}{Corollary}[section]
\newtheorem{prop}{Proposition}[section]
\newtheorem{lem}{Lemma}[section]
\newtheorem{defin}{Definition}[section]
\newtheorem{exmple}{Example}[section]
\newtheorem{rem}{Remark}[section]
\newtheorem{qz}{Question}[section]

\newcommand{\I}{\mathop {\rm I}\nolimits}

\newcommand{\III}{\mathop {\rm III}\nolimits}

\newcommand{\qed}{\hfill $\Box$}

\newcommand{\proof}{\noindent{\textsl {Proof}.}\hskip 3pt}

\renewcommand{\thesubparagraph}{\theparagraph.\@arabic\c@subparagraph}
\begin{document}
  
  \begin{center}
  
 {\bf  \Large 
On the topology of the complements of\\ reducible 
plane curves via Galois covers}
\bigskip

\bigskip
\large 
Shinzo BANNAI,
Masayuki KAWASHIMA
and
Hiro-o TOKUNAGA


\end{center}
\normalsize

\begin{abstract}
Let $\mcB$ be a reducible reduced plane curve. 
We introduce a new point of view to study the topology of $(\PP^2, \mcB)$ via Galois covers and 
Alexander polynomials. We show  its effectiveness through examples of Zariski $N$-plets for 
conic and  conic-quartic configurations.

\end{abstract}

{\large \bf Introduction}

 Let $\mcB$ be a reduced plane curve in the projective plane $\PP^2 (= \PP^2(\CC))$. Since
 Zariski's famous article \cite{zariski29}, the topology of the complement $\PP^{2}\setminus \mcB$ and the pair $(\PP^{2}, \mcB)$
 have been studied through various points of view by many mathematicians. Among them, there is an 
 approach via Galois covers of $\PP^{2}$ whose branch loci are contained in $\mcB$. In this article,
 we study the topology of $(\PP^{2}, \mcB)$ along this line. In previous articles by the third author
 (\cite{tokunaga94, tokunaga97, tokunaga98, tokunaga99, tokunaga04, tokunaga12}), 
 the  existence and non-existence problem of a Galois cover with fixed Galois group, fixed branched curve $\mcB$ 
 and fixed ramification type
  was considered and  results were applied  to study the topology of $(\PP^{2}, \mcB)$. In this article, however,
 we consider  various Galois covers with fixed Galois group, various branch curves contained 
 in $\mcB$ and various ramification types at the same time. This approach is first taken in \cite{bannai-tokunaga}.
 The  purpose of this article is to push forward the above idea and to study the topology of $(\PP^{2}, \mcB)$.
As a result, we obtain some new examples of  Zariski $N$-plets for conic-quartic configurations.

In order to explain our idea more precisely let us define some terminology.
 For normal projective varieties $X$ and $Y$ with
a finite morphism $\pi : X \to Y$, we say that $X$ is a Galois cover of $Y$ if the induced
field extension $\CC(X)/\CC(Y)$ is Galois, where $\CC(X)$ (resp. $\CC(Y)$) denotes the rational function 
field of $X$ (resp. $Y$). Under this circumstance,
 the Galois group $\Gal(\CC(X)/\CC(Y))$ acts on $X$ in a way such that $Y$ is the quotient
space with respect to this action ({\it cf.} \cite[\S 1]{tokunaga97}).
 If the Galois group 
$\Gal(\CC(X)/\CC(Y))$ is isomorphic to a finite group $G$,  $X$ is simply called a $G$-cover of $Y$.
The branch locus of $\pi : X \to Y$,  denoted by $\Delta_{\pi}$ or $\Delta(X/Y)$,
is a subset of $Y$ consisting of points $y$ of $Y$ such that $\pi$ is not locally
isomorphic over $y$. It is well-known that $\Delta_{\pi}$ is an algebraic subset of pure codimension $1$
if $Y$ is smooth (\cite{zariski}).  

 Now assume that $Y$ is smooth. Let $\mcB$ be a reduced divisor on $Y$ and denote its irreducible
 decomposition by $\mcB = \sum_{i=1}^r\mcB_i$.  
 A $G$-cover $\pi : X \to Y$ is said to be branched
 at $\sum_{i=1}^r e_i\mcB_i$ if $(i)$ $\Delta_{\pi} = \mcB$ (here we identify $\mcB$ with its support) and $(ii)$ the ramification index
 along $\mcB_i$ is $e_i \ge 2$ for each $i$, where  the ramification index means the
 one along the smooth part of $\mcB_i$ for each $i$. 
  
  Our approach via $G$-covers is based the following proposition which follows from
 the Grauert-Remmert theorem:
 
 \begin{prop}\label{prop:fund}{(\cite[Proposition 3.6]{act}) 
Under the notation as above, let $\gamma_i$ be a meridian around $\mcB_i$, and
$[\gamma_i]$ denote its class in the topological fundamental group 
$\pi_1(Y\setminus \mcB, p_o)$.
If there exists a $G$-cover $\pi : X \to Y$ branched at $e_1\mcB_1 + \cdots + e_r\mcB_r$, 
then there exists a normal subgroup $H_{\pi}$ of $\pi_1(Y\setminus \mcB, p_o)$ 
such that:

\begin{enumerate}
\item[(i)] $[\gamma_i]^{e_i} \in H_{\pi},  [\gamma_i]^k \not\in H_{\pi}, (1 \le k \le e_i -1)$, and
\item[(ii)]  $\pi_1(Y\setminus \mcB, p_o)/H_{\pi} \cong G$.
\end{enumerate}

Conversely, if there exists a normal subgroup $H$ of $\pi_1(Y\setminus \mcB, p_o)$ satisfying the above two conditions for $H_{\pi}$, then there exists a $G$-cover
$\pi_H : X_H \to Y$ branched at $e_1\mcB_1 + \cdots + e_r\mcB_r$.
}
\end{prop}

For $G$-covers $\pi_1 : X_1 \to Y$ and $\pi_2 : X_2 \to Y$, we identify them if there exists an isomorphism
$\Phi : X_1 \to X_2$ such that $\pi_1 = \pi_2\circ\Phi$.  
Under the same notation as in Proposition~\ref{prop:fund}, we introduce some additional 
notation and terminology. 

\begin{itemize}

\item For $\bde = (e_1, \ldots, e_r)$, a vector in $\ZZ^{\oplus r}_{\ge 1}$,
we say that a $G$-cover
 $\pi: X \to Y$ is branched at most along $\mcB$ of type $\bde$ if $(i)$
 $\Delta_{\pi}
= \Supp(\sum_i(e_i-1)\mcB_i)$ and $(ii)$ the ramification index along $\mcB_i$ is $e_i$.
Note that $\mcB_i$ is {\it not contained} in $\Delta_{\pi}$ if $e_i = 1$.

%


\item $\Cov_b(Y,\mcB(\bde), G)  :=$
the set of isomorphism classes of $G$-covers $\pi : X \to Y$ branched at most along $\mcB$ of type 
$\bde$.

\item $\Cov_b(Y, e_{i_1}B_{i_1} + \ldots + e_{i_s}B_{i_s}, G)  :=$
the set of isomorphism classes of $G$-covers $\pi : X \to Y$ branched at 
$e_{i_1}B_{i_1} + \cdots + e_{i_s}B_{i_s}$. Note that $\Cov_b(Y, e_{i_1}B_{i_1} + \ldots + e_{i_s}B_{i_s}, G)=\Cov_b(Y,\mcB(\bde^\prime), G)$ for $\bde' =(e^\prime_1,\ldots,e^\prime_r)$ where $e^\prime_j=1$ if $j\not\in\{i_1,\ldots,i_s\}$ and $e^\prime_j=e_j$ if $j\in\{i_1,\ldots,i_s\}$. 

\item $\bCov(Y,\mcB,G):=\bigcup_{\bde\in\ZZ_{\ge 1}^{\oplus r}} \Cov_b(Y,\mcB(\bde), G)$. This is the set of $G$-covers $\pi: X\to Y$ branched at most along $\mcB$ of \emph {all types}.


\end{itemize}
Sometimes, we omit $Y$ and/or  $G$ if they are clear from the context. 
Note again that  $\Cov_b(Y, \mcB(\bde), G)$ for fixed $\bde$ was considered in the third author's previous work. In this paper, we move a step further and consider  $\bCov(Y, \mcB, G)$ which will give us more information about $(Y, \mcB)$. We will consider $\bCov(Y, \mcB, G)$ with the additional information of the set of subsets $\{\Cov_b(Y,\mcB(\bde),G)\mid 
\bde\in\ZZ^{\oplus r}_{\geq 1}\}$ and denote it as 
\[{\bf Cov}(Y,\mcB,G):=\left(\bCov(Y, \mcB, G), \{\Cov_b(Y,\mcB(\bde),G) \mid  \bde\in \ZZ_{\geq 1}^{\oplus r}\}\right).\]
An equivalence relation can be defined on the set $\{{\bf Cov}(Y,\mcB,G)\mid \mcB:\text{ reduced divisor on $Y$ }\}$ as follows:
\begin{defin} 
Let $\mcB^k$ ($k=1,2$) be reduced divisors on $Y$.  Let $\mcB^k=\mcB^k_1+\cdots+\mcB^k_r$ be their irreducible decompositions.
 If there exists a bijection $f:\{\mcB^1_{1},\ldots,\mcB^1_{r}\}\rightarrow\{\mcB^2_{1},\ldots,\mcB^2_{r}\}$  and a bijection $g: \bCov(Y, \mcB^1, G)\rightarrow\bCov(Y,\mcB^2,G)$ such that $g(\Cov_b(Y, e_{i_1}B_{i_1}^1 + \ldots + e_{i_s}B_{i_s}^1 ,G))=\Cov_b(Y,e_{i_1}f(B_{i_1}^1) + \ldots + e_{i_s}f(B_{i_s}^1),G)$ for all $e_{i_1}B_{i_1}^1 + \ldots + e_{i_s}B_{i_s}^1 $, we say that ${\bf Cov}(Y, \mcB^1, G)$ and ${\bf Cov}(Y, \mcB^2, G)$ are equivalent and denote this by ${\bf Cov}(Y, \mcB^1, G)\approx{\bf Cov}(Y, \mcB^2, G)$.
\end{defin}
It can be easily checked that the relation $\approx$ is an equivalence relation. When we want to emphasize the bijection $f$ among the components, we use the symbol $\underset{f}{\approx}$. The following proposition is immediate from Proposition~\ref{prop:fund}:

\begin{prop}\label{prop:1-1}{
Let $\mcB^k$ ($k =1, 2$) be reduced divisors on $Y$. Let $\mcB^k=\mcB^k_1+\cdots+\mcB^k_r$ be their irreducible decompositions.
If there exists a homeomorphism $h: (Y, \mcB^1)\rightarrow(Y, \mcB^2)$, $h$ induces a bijection $\eta:\{\mcB^1_{1},\ldots,\mcB^1_{r}\}\rightarrow\{\mcB^2_{1},\ldots,\mcB^2_{r}\}$ and 
bijections  $h_\ast^{\bde} :\Cov_b(Y, e_{i_1}B_{i_1}^1 + \ldots + e_{i_s}B_{i_s}^1, G)\rightarrow\Cov_b(Y, e_{i_1}\eta(B_{i_1}^1) + \ldots + e_{i_s}\eta(B_{i_s}^1), G)$ for all $\bde\in\ZZ_{\geq1}^{\oplus r}$. 
In particular,  $h$ induces an equivalence ${\bf Cov}(Y,\mcB^1,G)\underset{\eta}{\approx}{\bf Cov}(Y, \mcB^2,G)$.}
\end{prop}

This set up can be interpreted in the following way. For each $I=\{i_1,\ldots,i_n\}\subset\{1,\ldots,r\}$, let $\mcB_{I}=\mcB_{i_1}+\cdots+\mcB_{i_n}$. Furthermore let ${\rm Sub}(\mcB):=\{\mcB_I|I\subset\{1,\ldots,r\}\}$ be the set of sub-configurations of $\mcB$. A homeomorphism $h$ as in Proposition \ref{prop:1-1} induces a inclusion preserving bijection $\eta: {\rm Sub}(\mcB^1)\rightarrow{\rm Sub}(\mcB^2)$, i.e.  if $\mathcal{B}_{I^\prime}^1\subset\mathcal{B}_I^1$ then $\eta(\mcB_{I^\prime}^1)\subset\eta(\mcB_{I}^1)$. Furthermore, ${\bf Cov}(Y,\bullet, G)$ can be viewed as a map from ${\rm Sub}(B^k)$ to $\{{\bf Cov}(Y,\mcB,G)\mid \mcB:\text{ reduced divisor on $Y$ }\}$. Then we have the following commutative diagram:
\[
\begin{diagram}
\node{{\rm Sub}(\mcB^1)}\arrow{se,t}{{\bf Cov}(Y,\bullet,G)}\arrow{s,l}{\eta} \\
\node{{\rm Sub}(\mcB^2)}\arrow{e,b}{ {\bf Cov}(Y,\bullet,G)}\node{\{{\bf Cov}(Y,\mcB,G) \mid  \mcB:\text{ reduced divisor on $Y$ } \}} 
\end{diagram}
\]

Summing up  we have:

\begin{cor}\label{key-cor1}
Under the above setting ${\bf Cov}(Y,\mcB_I,G)\underset{\eta}{\approx}{\bf Cov}(Y, \eta(\mcB_I),G)$ for all $\mcB_I\in{\rm Sub}(\mathcal{B}^1)$. Conversely, if such $\eta$ compatible with ${\bf Cov}(Y,\bullet,G)$  in this sense does not exist, $(\PP^2,\mcB^1)$ and $(\PP^2,\mcB^2)$ are not homeomorphic. 
\end{cor}

We can replace  ${\bf Cov}(Y,\bullet,G)$  with various different topological invariants and conduct a similar discussion. In particular, we will  also study the case of Alexander polynomials. 
Let  $Y=\PP^2$ and $\mcB$ be a reduced plane curve. Let $\Delta_\mcB(t)$ be the Alexander polynomial of $\mcB$ (see Section \ref{alex} for the definition) and  let 
$\bAlex_{\mcB}: {\rm Sub}(\mcB)\rightarrow\CC[t]$ be defined by $\bAlex_{\mcB}(\mathcal{B}_I)=\Delta_{\mcB_I}(t)$. Now, if a homeomorphism $h: (\PP^2,\mcB^1)\rightarrow(\PP^2,\mcB^2)$ exists, it induces a bijection $\eta$ as before such that is compatible with $\bAlex_{\mcB^k}$. 

\begin{cor}\label{key-cor2}
Under the above setting $\bAlex_{\mcB^1}=\bAlex_{\mcB^2}\circ\eta$. Conversely if such $\eta$ does not exist, $(\PP^2,\mcB^1)$ and $(\PP^2,\mcB^2)$ are not homeomorphic. 
\end{cor}

In \cite{bannai-tokunaga}, we make use of a subset of $\bCov(\PP^{2}, \mcB, D_{2p})$, where 
$\mcB$ is a conic arrangement consisting of $k+2$ conics $\mcC^{'} + \mcC^{''} + \mcC_{1}+\ldots +\mcC_{k}$
and $D_{2p}$ is the dihedral group of order $2p$ ($p$: odd prime) to construct Zariski $N$-plets for
conic arrangements. In fact, we study 
$\Cov_{b}(\PP^{2}, 2(\mcC^{'} + \mcC^{''}) 
+ p(\mcC_{i} + \mcC_{j}), D_{2p})$
$(1 \le i < j \le k)$.  In this article, we consider certain conic-quartic configurations $\mcB$ and show that
considering the whole of ${\bf Cov}(\PP^2, \mcB, G)$ is more effective. We also compute Alexander polynomials
 of the sub-cofigurations
 of $\mathcal{B}$. 

This article goes as follows. In \S 1, we summarize known results on $D_{2n}$-covers and  the theory of elliptic surfaces, which play important roles in our construction of examples. 
 In \S 2, we review the definition of Alexander polynomials and compute them for the configurations we consider.  In \S 3, we compare
${\bf Cov}(\PP^{2}, \mcB, G)$ with $\bAlex_{\mcB}$ and show that the former is a finer invariant than the latter through examples.




\section{$D_{2p}$-covers and elliptic surfaces}

\subsection{$D_{2n}$-covers}

We here introduce notation for dihedral covers.
Let $D_{2n}$ be the dihedral group of order $2n$. 
 In order to present $D_{2n}$, we use
 the notation
 \[
 D_{2n} = \langle \sigma, \tau \mid \sigma^2 = \tau^n = (\sigma\tau)^2 = 1\rangle.
 \]
 Given a $D_{2n}$-cover, we obtain a double cover  $D(X/Y)\to Y$ of $Y$ 
   canonically, by considering the
 $\CC(X)^{\tau}$-normalization of $Y$, where $\CC(X)^{\tau}$ denotes the fixed field 
 of the subgroup generated by $\tau$.  Here, $X$ is an $n$-fold cyclic cover of $D(X/Y)$ and
 we denote these covering morphisms by
 $\beta_1(\pi) : D(X/Y) \to Y$ and $\beta_2(\pi) : X \to D(X/Y)$, respectively.  

\subsection{Elliptic Surfaces}\label{elliptic}
We first summarize some facts from the theory of elliptic surfaces. As for details, we refer to 
\cite{kodaira}, \cite{miranda-basic}, \cite{miranda-persson} and \cite{shioda90}.

 In this article, by an {\it elliptic surface}, we always mean a smooth projective surface $S$ 
  with a fibration $\varphi : S \to C$ over a smooth projective curve $C$, satisfying the following:
 \begin{enumerate}
  \item[(i)] There exists a non-empty finite subset $\Sing(\varphi)$, of $C$, such
  that $\varphi^{-1}(v)$ is a smooth curve of genus $1$ (resp. a singular curve) for $v \in C\setminus \Sing(\varphi)$ (resp. $v \in \Sing(\varphi))$.
 \item[(ii)]  $\varphi$ has a section
 $O : C \to S$ (we identify $O$ with its image). 
 \item[(iii)]  $\varphi$ is minimal, i.e., there is no exceptional
 curve of the first kind in any fiber. 
 \end{enumerate}
 
 For $v \in \Sing(\varphi)$, we put $F_v = \varphi^{-1}(v)$. 
 We denote its irreducible decomposition by 
 \[
 F_v = \Theta_{v, 0} + \sum_{i=1}^{m_v-1} a_{v,i}\Theta_{v,i}, 
 \]
 where $m_v$ is the number of irreducible components of $F_v$ and $\Theta_{v,0}$ is the
 irreducible component with $\Theta_{v,0}O = 1$. We call $\Theta_{v,0}$ the {\it identity
 component}.  The classification  of singular fibers is well known (\cite{kodaira}). 
 Note that every
 smooth irreducible component of  reducible singular fibers is a rational curve with self-intersection number $-2$.

 We also define a subset of $\Sing(\varphi)$ by
 $\Red(\varphi) := \{v \in \Sing(\varphi)\mid \mbox{$F_v$ is reducible}\}$. Let 
 $\MW(S)$ be the set of sections of $\varphi : S \to  C$. From our assumption, 
 $\MW(S) \neq  \emptyset$. By regarding $O$ as the zero element of 
 $\MW(S)$ and 
 considering fiberwise addition (see \cite[\S 9]{kodaira} or \cite[\S 1]{tokunaga11} for
 the addition on singular fibers), 
 $\MW(S)$ becomes an abelian group. We denote its addition by $\dot{+}$.
 Note that the ordinary  $+$ is used for the sum of divisors, and the two operations should not be confused. 
 
Let $\NS(S)$ be the N\'eron-Severi group of $S$ and let $T_{\varphi}$ be the 
subgroup of $\NS(S)$ generated by $O, F$ and  $\Theta_{v,i}$ $(v \in \Red(\varphi)$,
$1 \le i \le m_v-1)$. Then we have the following theorems:

\begin{thm}\label{thm:shioda-basic0}{\cite[Theorem~1.2]{shioda90}) Under our assumption,
$\NS(S)$ is torsion free.
}
\end{thm}

\begin{thm}\label{thm:shioda-basic}{(\cite[Theorem~1.3]{shioda90}) Under our assumption,
there is a natural map $\tilde{\psi} : \NS(S) \to \MW(S)$ which induces an isomorphisms of 
groups
\[
\psi : \NS(S)/T_{\varphi} \cong \MW(S).
\]
In particular, $\MW(S)$ is a finitely generated abelian group.
}
\end{thm}

In the following,  by the rank of $\MW(S)$, denoted by $\rank\MW(S)$, we mean that of the free part 
of $\MW(S)$.
 For a divisor on $S$, we  put $s(D) = \psi(D)$.  As for the relation between $D$ and $s(D)$, see 
 \cite[Lemma~5.1]{shioda90}.
 Also,  in \cite{shioda90},  a $\QQ$-valued bilinear form $\langle \, , \, \rangle$ on
 $\MW(S)$  is defined by using the intersection pairing on $\NS(S)$.   Here are two basic properties of $\langle \, , \, \rangle$:

\begin{itemize}

\item $\langle s, \, s \rangle \ge 0$ for $\forall s \in \MW(S)$ and the equality holds if and 
only if $s$ is an element of finite order in $\MW(S)$. 

\item An explicit formula for $\langle s_1,
s_2\rangle$ ($s_1, s_2 \in \MW(S)$) is given as follows:
\[
\langle s_1, s_2 \rangle = \chi({\mathcal O}_S) + (s_1, O) +( s_2, O) -( s_1, s_2) - \sum_{v \in \Red(\varphi)}
\mbox{Contr}_v(s_1, s_2),
\]
where $(,)$ denotes the intersection pairing of divisors and $\mbox{Contr}_v(s_1, s_2)$ is given by
\[
\mbox{Contr}_v(s_1, s_2) = ((s_1, \Theta_{v,1}), \ldots, (s_1, \Theta_{v, m_v-1}))(-A_v)^{-1}
\left ( \begin{array}{c}
        (s_2, \Theta_{v,1}) \\
        \vdots \\
        (s_2, \Theta_{v, m_v-1})
        \end{array} \right ).
\]
 As for explicit values of 
$\mbox{Contr}_v(s_1, s_2)$, we refer to \cite[(8.16)]{shioda90}.
\item Let $\MW(S)^0$ be the subgroup of $\MW(S)$ given by 
\[
\MW(S)^0 := \{ s \in \MW(S) \mid \mbox{$s$ meets $\Theta_{v, 0}$ for $\forall v \in \Red(\varphi)$.}\}
\]
$\MW(S)^0$ is called the {\it narrow part} of $\MW(S)$. By the explicit formula as above,
$(\MW(S)^0, \, \langle \, , \, \rangle)$ is a positive definite even integral lattice.
\end{itemize}

\subsection{Rational elliptic surfaces of admissible type and conics}\label{admissible}

 Let $\mathcal{Q}\subset\PP^2$ be a reduced quartic curve having at least one non-linear component, $P\in \mathcal{Q}$ be a general point on a non-linear component of $\mathcal{Q}$, $\Lambda_P$ be the pencil of  lines through $P$. Let $S^{\prime\prime}$ be the double cover of $\PP^2$ branched along $\mathcal{Q}$, and $S$ be the canonical resolution of  singularities of $S^{\prime\prime}$. The pencil $\Lambda_P$ lifts to a pencil of genus 1 curves in $S$, and by resolving the base points of this pencil, we obtain an elliptic surface which we denote by $S_P$.   We will denote the composition of all the morphisms by $f_P:S_P\to \PP^2$. The exceptional divisor of the final blow-up in resolving the pencil will be denoted by $O$ and we will consider it as the zero-section. There exists a group structure on the generic fiber of $\varphi: S \rightarrow \PP^1$ where $O$ is considered as  the zero element.  Note that the involution of the double cover lifts to the elliptic surface and the lift  coincides with the involution of the elliptic surface defined by taking the inverse in the above group structure.  The singular fibers  of  $S_P$  correspond to certain members of $\Lambda_P$ which are the tangent line of $\mathcal{Q}$ at $P$, and the lines passing through singular points of $\mathcal{Q}$. We will denote the components of the singular fiber corresponding to the tangent line by $\Theta_{0,i}$. The components of the other singular fibers will be denoted by $\Theta_{v,i}$ where $v\in \Sing(\mathcal{Q})$. In both cases the component that intersects with $O$ will be labeled by $i=0$. The resolution of base points of $\Lambda_P$ consists of two consecutive blow ups, and the strict transform of the exceptional divisor of the first blow up becomes $\Theta_{0,0}$ and the exceptional divisor of the second blow up becomes the zero-section $O$. The components of the other singular fibers consist of  the exceptional divisors 
of the resolution of singularities and the strict transforms of members of $\Lambda_P$  passing through the singular points of $\mathcal{Q}$. The latter of these components are the components that intersect $O$,  i.e. they are the components labeled $\Theta_{v,0}$. 

\begin{lem}\label{conic}
Under the setting given above,  let $s$ be a section of height equal to 2 contained in $\MW(S_P)^0$. Then $f_P(s)$ is a smooth conic such that (i) $P\in f_P(s)$, (ii) $f_P(s)$ and $\mathcal{Q}$ intersect at smooth points of $\mathcal{Q}$ and (iii) the local intersection multiplicities of $f_P(s)$ and $\mathcal{Q}$ are all even.
\end{lem}
\proof First we note that by the explicit formula for the height pairing we have
\[
\langle s, s \rangle=2+(s, O)-\sum{\rm contr}_v(s)
\]
Since $s\in \MW(S)^0$ and the height of $s$ is 2, we have $(s, O)=0$. Similarly we have $(-s,O)=0$. Since $\langle s, -s \rangle=-2$, and from the explicit formula 
\[
\langle s, -s \rangle=1+(s, O)+(-s, O)-(s, -s)-\sum{\rm contr}_v(s)
\]
we have $(s, -s)=3$. After blowing down $O$ and $\Theta_{0,0}$, in this order, the strict transform $\bar{s}$ (resp. $\overline{-s}$) of $s$ (resp. $-s$)  on $S$ satisfy $P\in\bar{s}\cap\overline{-s}$, $\bar{s}^2=\overline{-s}^2=0$, $(\bar{s}, \overline{-s})=4$. Hence $(\bar{s}+\overline{-s}, \bar{s}+\overline{-s})=8$.  Since $s\in \MW(S_P)^0$, $s$ does not intersect any exceptional divisors of the resolution of singularities, so we have $(f_P(s),f_P(s))=4$, hence $f_P(s)$ is a irreducible conic. The condition on the local intersection numbers follow because $f_P(s)$ is the image of $\bar{s}+\overline{-s}$. 

An elliptic surface $S$ such that $\MW(S)^0$ is isomorphic to a root lattice, is called an elliptic surface of admissible type. This terminology is due to Shioda (\cite{shioda92}).

\begin{cor}
Suppose $S_P$ is an rational elliptic surface of admissible type. Let $s$ be a root in  $\MW(S_P)^0$. Then $\mathcal{C}=f_P(s)=f_P(-s)$ is a conic satisfying the properties in Lemma \ref{conic}.
\end{cor}
\proof We only have to note that all root lattices that appear as a narrow Mordell-Weil lattice are combinations of root lattices of type $A$, $D$, $E$, whose roots have height equal to $2$.\qed

By varying the point $P$, we can construct families of conics that are tangent to $\mathcal{Q}$. 

\begin{cor}\label{subscribed}
Let $S_P$ be  rational elliptic surface of admissible type. Let $\{\pm s_1,\ldots, \pm s_r\}$ be the set of roots of $\MW(S_P)^0$. Then there exist families of conics $\mathcal{F}_1, \ldots, \mathcal{F}_r$ in $\PP^2$ such that if  $\mathcal{C}\in \mathcal{F}_i$ the following hold: 
\begin{enumerate}
 \item $\mathcal{C}$ intersects $\mathcal{Q}$ at smooth points of $\mathcal{Q}$,
 \item  the local intersection multiplicity at each intersection point is even, and
 \item  the inverse image of $\mathcal{C}$ in $S_P$ consists of two irreducible components $f_P^{-1}(\mathcal{C})=\mcC^+\cup \mcC^-$ such that $\{\tilde{\psi}_P(\mcC^+), \tilde{\psi}_P(\mcC^-)\}=\{s_i, -s_i\}$.
 \end{enumerate}
\end{cor}
\proof The only thing new here is the final statement $\{\tilde{\psi}_P(\mcC^+), \tilde{\psi}_P(\mcC^-)\}=\{s_i, -s_i\}$ which follows from Theorem 1 in \cite{bannai-tokunaga}.\qed

\subsection{Elliptic $D_{2p}$-covers}
We keep the same notation as the previous subsection. In \cite{tokunaga12}, the third author gave a criterion for the existence of dihedral covers branched at $\mathcal{Q}$ and some additional loci, in terms of elliptic surfaces. The criterion, rephrased in our situation here,  is  as follows:
\begin{lem}[\cite{tokunaga12} Theorem 3.1]\label{criterion}
Let $p$ be an odd prime number. Let $\mathcal{C}$ be a reduced possibly reducible curve. Let $\mathcal{C}=\mathcal{C}_1+\cdots+\mathcal{C}_n$ be its irreducible decomposition. Suppose that $f^\ast_P(\mathcal{C}_i)$ is reduced and horizontal in $S_P$ and $f_P^\ast(\mathcal{C}_i)=\mathcal{C}_i^++\mathcal{C}_i^-$, i.e. it  splits into two irreducible components. Then the following statements are equivalent.
\begin{itemize}
\item There exists a $D_{2p}$-cover of $\PP^2$ branched at $2\mathcal{Q}+p\mathcal{C}$.
\item There exist positive integers $a_i$, $1\leq a_i<p$ such that $\sum_{i=1}^n a_i{\tilde{\psi}}(\mathcal{C}_i^+)\in p\MW(S_P)$ where
$p\MW(S_P):=\{pv\mid v\in\MW(S_P)\}$. 
\end{itemize} 
\end{lem}
Since Corollary \ref{subscribed} allows us to construct families of conics with prescribed images under $\tilde{\psi}$, combining with the criterion above, we have a method of constructing various configurations that are or are not branch curves of $D_{2p}$-covers of $\PP^2$.

\section{Alexander polynomials of certain reducible curves}\label{alex}

\subsection{Alexander Polynomials}

Let $\mathcal{C}\subset\CC^2$ be an affine curve of degree $d$.
Suppose that the line at infinity $L_{\infty}$
intersects transversely with $\mathcal{C}$. Put $X:=\CC^2\setminus\mathcal{C}$.
Let $\phi:\pi_1(X)\to \Bbb{Z}$ be 
the composition of Hurewicz homomorphism and the summation homomorphism.
Let $t$ be a generator of $\Bbb{Z}$ and 
$\Lambda:=\Bbb{C}[t,t^{-1}]$.
We consider an infinite cyclic covering $p:\tilde{X}\to X$ such that
$p_{*}(\pi_1(\tilde{X}))=\ker \phi$.
Then $H_1(\tilde{X},\Bbb{C})$ has a structure of a $\Lambda$-module. 
Thus we have
\[
 H_1(\tilde{X},\Bbb{C})=\Lambda/\lambda_1(t)\oplus\cdots \oplus\Lambda/\lambda_m(t)
\]
where we can take $\lambda_i(t)\in \Lambda$ as a polynomial in $t$ such that
$\lambda_i(0)\ne 0$ for $i=1,\dots, m$.
The Alexander polynomial $\Delta_\mathcal{C}(t)$ is defined  by the product
$\prod_{i=1}^m \lambda_i(t)$.

In this paper,
 we use the {\it{Loeser-Vaqui\'e formula}} (\cite{OkaSurvey,OkaAtlas})
for calculating Alexander polynomials.
Hereafter we follow the notation and terminology of
\cite{Kawashima2,OkaAtlas} for the Loeser-Vaqui\'e formula.

 \subsubsection{{\rm{\bf{Loeser-Vaqui\'e formula}}}}
Let $[X,Y,Z]$ be homogenous coordinates of $\Bbb{P}^2$ and
let   $(x,y)=(X/Z,X/Z)$ be affine coordinates of
$\Bbb{C}^2=\Bbb{P}^2\setminus \{Z=0\}$.
Let $f(x,y)$ be the defining polynomial of $\mathcal{C}$ of degree $d$.
%
 Let $\Sing(\mathcal{C})$ be the set of singular points  of $\mathcal{C}$
and let $P\in \Sing(\mathcal{C})$ be a singular point.
 Consider a 
 resolution $\pi:\tilde{U}\to U$ of $(\mathcal{C},P)$, 
 and let $E_1,\dots,E_s$ be the exceptional divisors of $\pi$.
Let $(u,v)$ be a local coordinate system centered at $P$
and $k_i$, $m_i$ be  respective  orders of zero of the canonical two form 
$\pi^{*}(du \wedge dv)$ and $\pi^{*}f$ along the divisor $E_i$.
{\em{The adjunction ideal $\mathcal{J}_{P,k,d}$ of $\mathcal{O}_P$}} 
 is defined by 
\[
\mathcal{J}_{P,k,d} 
=\left\{\phi\in \mathcal{O}_P \ |\  (\pi^{*}\phi)\ge
\sum_{i}(\lfloor km_i/d\rfloor-k_i)E_i\right\},
 \quad k=1,\dots,d-1
\]
where $\lfloor * \rfloor=\max \{ n\in \Bbb{Z}\mid n \le *\}$,  the floor function.

Let $O(j)$ be the set of polynomials in $x,y$ 
whose degree is less than or equal to $j$.
We consider the canonical mapping 
$\sigma:\,{\Bbb C}[x,y]\to \bigoplus_{P\in \Sing(\mathcal{C})}\mathcal{O}_P$ and its restriction:
\[
\sigma_{k}:\,O(k-3)\to \bigoplus_{P\in \Sing(\mathcal{C})}\mathcal{O}_P.
\]
Put $V_k(P):=\mathcal{O}_{P}/\mathcal{J}_{P,k,d}$ and
denote the composition of $\sigma_k$ and the natural surjection 
$\bigoplus\mathcal{O}_P\to \bigoplus V_k(P)$
by $\bar{\sigma}_{k}$.
Then the Alexander polynomial of $C$ is given by the following formula called the Loeser-Vaqui\'e formula:

\begin{thm}\label{Loeser-Vaquie}
$(\cite{MR85h:14017,Loeser-Vaquie,Artal,Esnault})$
The reduced Alexander polynomial $\tilde{\Delta}_\mathcal{C}(t)$ is given by 
\[\tag{1}\label{A1}
\tilde{\Delta}_{\mathcal{C}}(t)=\prod_{k=1}^{d-1}\Delta_k(t)^{\ell_k},\quad
  \ell_k:=\dim {\rm coker}\, \bar{\sigma}_k
 \]
 where
\[
 \Delta_k(t)=\left(t-\exp\left(\dfrac{2k\pi i}{d}\right)\right)
\left(t-\exp\left(-\dfrac{2k\pi i}{d}\right)\right).
\]
  The Alexander polynomial $\Delta_{\mathcal{C}}(t)$ is given as
\[
 \Delta_{\mathcal{C}}(t)=(t-1)^{r-1}\tilde{\Delta}_{\mathcal{C}}(t)
\]
where $r$ is the number of irreducible components of $\mathcal{C}$.
\end{thm}

For $A_1$ and $A_3$ singularities, the adjunction ideal can be calculated as follows. For the details of the calculations see \cite{OkaAtlas} \S 2.1 Lemma 2.

\begin{lem}\label{egideal}
 Let $\mathcal{C}: f=0$ be a plane curve of degree $2n+4$ such that 
 $\mathcal{C}$ has only $A_1$ and $A_3$ singularities.
 Let $P$ be a singular point.
\begin{enumerate}
\item[$(1)$]  Assume that  $P$ is an $A_1$ singularity.
 Then the adjunction ideal is  
\[
 \mathcal{J}_{P,k,2n+4}=
 \left\langle u^av^b\mid a+b\ge  \left\lfloor \frac{k}{n+2} \right\rfloor-1\right\rangle=\mathcal{O}_P,
\]
for all $k=3,\dots,2n+3$.
Hence $A_1$ singularities do not contribute to the computation
of Alexander polynomials
because $V_P(k)$ is $0$ for all $k$.
\item[$(2)$] 
Assume that $P$ is an $A_3$ singularity.
              Then the adjunction ideal is  
  \[
 \mathcal{J}_{P,k,2n+4}=\left\langle u^av^b\mid a+2b\ge
 \left\lfloor \frac{2k}{n+2} \right\rfloor -2\right\rangle=\begin{cases}
                                                            \mathcal{O}_P,
                                                            & 3\le k<
                            \left\lceil\dfrac{3}{2}(n+2) \right\rceil,\\
                                            m_P, &
                                             \left\lceil\dfrac{3}{2}(n+2) \right\rceil\le k\le
                                             2n+3
                                         \end{cases}
\]
             where  $\lceil * \rceil=\min \{ n\in \Bbb{Z}\mid * \le n\}$, the ceil function.
 
 \end{enumerate}
\end{lem}

\subsection{Computation of Alexander polynomials}

Let ${\rm Sub}(\mcB)_s=\{\mcB_I \mid |I|=s\}$, the set of subconfigurations consisting of $s$ irreducible components.
Define the map $\widetilde{\bAlex}_{\mcB}$ from ${\rm Sub}(\mcB)$ to $\CC[t]$ by
\begin{align*}
\widetilde{\bf{Alex}}_{\mcB}(\mcB_I) :=  \tilde{\Delta}_{\mcB_I}(t)
\end{align*}
Let $\bAlex_{\mcB,s}=\bAlex|_{{\rm Sub}(\mcB)_s}$ and $\widetilde{\bAlex}_{\mcB,s}=\widetilde{\bAlex}|_{{\rm Sub}(\mcB)_s}$.
We say that ${\widetilde{\bf{Alex}}_\mcB}$ (resp. ${\widetilde{\bf{Alex}}_{\mcB,s}}$) is  {\em{trivial}} on $\mathcal{B}$  if
$ {\widetilde{\bf{Alex}}_\mcB}(\mcB_I)=1$ for all $\mcB_I\in{\rm Sub}(\mcB)$ (resp. $\mcB_I\in{\rm Sub}_s(\mcB))$.

In this subsection, we compute the Alexander polynomial  for a certain type of a reduced curve $\mcB$.
 Let $\mathcal{Q}$ be a reduced quartic.
Suppose that $\mathcal{Q}$ has at most $A_1$ singularities.
Let $\mathcal{C}_1,\dots,\mathcal{C}_n$ be smooth conics such that:
    \begin{enumerate}
     \item Each $\mathcal{C}_i$ is tangent to $\mathcal{Q}$ with intersection multiplicity $2$
           at $4$ smooth points for any $i$.
     \item For all pairs $(i,j)$ $(i\ne j)$,
           $\mathcal{C}_i$ intersects transversely with $\mathcal{C}_j$ at all intersection
           points.
           
     \item $\mathcal{C}_i\cap \mathcal{C}_j\cap
           \mathcal{C}_\ell=\emptyset$.
     \end{enumerate}
Let $\mathcal{B}:=\mathcal{Q}+\mathcal{C}_1+\cdots+\mathcal{C}_n$. 
 Let $\mathcal{Q}\cap\mathcal{C}_i=\{P_{i1},\dots,P_{i4}\}$ be
 the tangent points of $C_i$ and $Q$ for $i=1,\dots,n$.
 Note that the configurations of singularities of $\mathcal{B}$ is
 \[
  \Sigma_{{\rm top}}(\mathcal{B})=\Sigma_{{\rm top}}(\mathcal{Q})+\{2n(n-1)A_1,4nA_{3}\},
 \]
where $\Sigma_{\rm top}(\bullet)$ is the set of topological types of the singularities of $\bullet$. For this curve,
${\widetilde{\bf{Alex}}}_{\mcB,1}$ and
${\widetilde{\bf{Alex}}}_{\mcB,2}$ are trivial.

 Now we consider ${\widetilde{\bf{Alex}}}_{\mcB,s}$ where $s\ge 3$.
Let $I$ be a non-empty subset of $\{1,\dots,n+1\}$.
We suppose that the index  $n+1$ correponds to the quartic $\mathcal{Q}$.
If $n+1\not\in I$ that is $\mcQ$ is not contained in $\mcB_I$, then
 $\tilde{\Delta}_{\mathcal{\mathcal{B}}_I}(t)=1$ as $\mathcal{B}_{I}$ has only $A_1$ singularities.
 Hence we consider the Alexander polynomial of
 $\mathcal{B}_I$ where $I$ contains $n+1$. Then $\mcB_I$ will have the same configuration as $\mcB$ except for the number of conics involved. So by relabeling, computing 
 $\Delta_{\mcB}(t)$ (for $n=|I|-1$) is enough.

To determine the Alexander polynomial of $\mathcal{B}$,
we consider the adjunction ideals and the map 
$\bar\sigma_k:O(k-3)\to V(k)$.
The adjunction ideal for each singular point
is given in Lemma \ref{egideal}.
Now we consider the multiplicity $\ell_k$ in the formula  (\ref{A1}) of Loeser-Vaqui\'e
which is given as
\[
\ell_k=\dim {\rm coker\,} \bar{\sigma}_k
={\rho}(k)+\dim \ker \bar{\sigma}_k.
\]
where
$\rho (k)=\sum_{P\in \Sing(\mathcal{B}_I)}\dim V_k(P)-\dim O(k-3)$.
For fixed $k$, the integer $\rho (k)$ is determined
 only by the adjunction ideal.
 Hence we should consider the dimension of $\ker \bar{\sigma}_k$.

By Lemma \ref{egideal},
if $k<\left\lceil \frac{3}{2}(n+2)\right\rceil$,
then $V(k)=0$. That is $\ell_k=0$.
 For the other cases, $V(k)=\Bbb{C}^{4n}$ and
$g\in \ker \bar\sigma_k$ if and only if 
$\{g=0\}$ passes through all of the $A_3$ singular points.
Hence we should investigate the linear series consisting of curves of degree $k-3$ passing through the points $P_{ij}$, $1\leq i\leq n,1\leq j\leq 4$. In the following, we denote the  linear series of plane curves of degree $d$ passing through points $P_{i}$ by $\mathcal{L}_{d}(-\sum P_{i})$.
 In general,
 for the $4n$ points $P_{ij}$, the dimension of $\mathcal{L}_{k-3}\left(-\sum P_{ij}\right)$ is greater
than or equal to $N:=\frac{(k-2)(k-1)}{2}-4n-1$.
Note that if $\dim \mathcal{L}_{k-3}(-\sum P_{ij})=N$,
then $\ell_{k}=0$.

\begin{lem}\label{2n+3}
 If $k=2n+3$, then
 $\dim \mathcal{L}_{2n}(-\sum P_{ij})=N$. Furthermore, $\ell_k=0$.
  \end{lem}

  \begin{proof}
Assume 
   $\dim \mathcal{L}_{2n}(-{\sum P_{ij}})\ge N+1$.
 Choose   $2n^2-n$ distinct points $Q_{ij}$, $1\leq i\leq n$, $1\leq j\leq 4i-3$ so that $Q_{ij}\in\mathcal{C}_{i}\setminus\Sing(\mcB)$. Choose one additional point $R\not\in\mathcal{C}_1$. Then $\sharp(\{Q_{ij}\}\cup\{R\})=2n^2-n+1=N+1$ and $\dim \mathcal{L}_{2n}(-\sum P_{ij}-\sum Q_{ij}-R)\geq\dim\mathcal{L}_{2n}(-\sum P_{ij})-(N+1)\geq 0$.    Hence we can take an element
   $\mathcal{D}\in \mathcal{L}_{2n}(-\sum P_{ij}-\sum Q_{ij}-R)$.
As the points $Q_{nj}$ are distinct from $P_{nj}$,
   we have $I(\mathcal{D},\mathcal{C}_n)\ge (4n-3)+4=2\cdot 2n+1$.
   Hence $\mathcal{D}$ must contain $\mathcal{C}_n$ as an irreducible component, therefore  $\mathcal{D}=\mathcal{D}_1+\mathcal{C}_n$ for some  $\mathcal{D}_1\in \mathcal{L}_{2n-2}\left(-\sum_{i=1}^{n-1}\left(\sum_{j=1}^4 P_{ij}\right)-\sum_{i=1}^{n-1}\left(\sum_{j=1}^{4i-3} Q_{ij}\right)-R\right).$
  By the same argument for  $\mathcal{C}_i$ with  $i=n-1,\ldots,2$, $\mathcal{D}=\mathcal{D}_n+\mathcal{C}_{2}+\cdots+\mathcal{C}_{n}$ where $\mathcal{D}_n\in\mathcal{L}_2(-\sum_{j=1}^4 P_{1j}-Q_{11}-R)$, but $\mathcal{L}_2(-\sum_{j=1}^4 P_{1j}-Q_{11}-R)=\emptyset$
   because $R\notin C_1$.
   This is a contradiction.
\end{proof}

\begin{lem}\label{2n+2}
 If $n\ge 3$ and $k=2n+2$,
 then $\dim \mathcal{L}_{2n-1}(-\sum P_{ij})=N$. Furthermore, $\ell_k=0$.
  \end{lem}
 \begin{proof}
We assume  $\dim \mathcal{L}_{2n-1}(-\sum P_{ij})\ge N+1$.
We divide our considerations into two cases
  $\mathcal{L}_2(-\sum_{j=1}^4 {P}_{i_1j}-\sum_{j=1}^4 P_{i_2j})=\emptyset$ for all $(i_1,i_2)$ or
  $\dim \mathcal{L}_2(-\sum_{j=1}^4 {P}_{i_1j}-\sum_{j=1}^4 P_{i_2j})\not=\emptyset$ for some $(i_1,i_2)$. 
   The first case can be proved by the same argument as Lemma \ref{2n+3}.

Now we consider the second case.
We may assume that $(i_1,i_2)=(1,2)$ and we take
a  conic $\mathcal{C}_0 \in \mathcal{L}_2(-\sum P_{1j}-\sum P_{2j})$.
First, choose $2n^2-7n+6$ distinct points $Q_{ij}$, $3\leq i\leq n$, $1\leq j\leq 4i-9$ so that $Q_{ij}\in \mathcal{C}_i\setminus\Sing(\mcB)$.
Choose an additional $4n-9$ distinct points $Q_{0j}$ $1\leq j \leq 4n-9$ so that  $Q_{0j}\in \mathcal{C}_0\setminus \mcB$.
Let $R_1, R_2, R_3\not\in \mathcal{C}_0\cup\mcB$ be noncollinear points. Then we have a total of $2n^2-3n=N+1$ points.

 Then 
 $\dim \mathcal{L}_{2n-1}(-\sum P_{ij}-\sum Q_{ij}-\sum R_i)\ge
   \dim\mathcal{L}_{2n-1}(-\sum P_{ij})-(N+1)=0$.
   Hence we can take an element
   $\mathcal{D}\in  \mathcal{L}_{2n-1}(-\sum P_{ij}-\sum Q_{ij}-\sum R_i)$.
By the same argument as before, $\mathcal{D}$ can be decomposed as 
  $\mathcal{D}=\mathcal{D}_{n-1}+\mathcal{C}_3+\cdots+\mathcal{C}_n+\mathcal{C}_0$ where $\mathcal{D}_{n-1}\in\mathcal{L}_1(-\sum R_i)$.
   But $\mathcal{L}_1(-\sum R_i)=\emptyset$ because
  $R_1,R_2$ and $R_3$ are not collinear. 
  
  \end{proof}

   \begin{prop}\label{co1}
${\widetilde{{\bf{Alex}}}}_{\mcB,4}$ is trivial on $\mcB$.
   \end{prop}
\begin{proof}
Let $\mcB_I$ be a sub-configuration of the form $\mcB_I=\mcQ+\mcC_{i_1}+\mcC_{i_2}+\mcC_{i_3}$. The  arguments before Lemma \ref{2n+3} for $n=3$ tells us for $1\leq k<8$, $V(k)=0$ which implies $\ell_1=\cdots=\ell_7=0$, and Lemma \ref{2n+3} and \ref{2n+2} tell us $\ell_8=\ell_9=0$. Hence by Theorem \ref{Loeser-Vaquie}, $\tilde{\Delta}_{\mcB_I}(t)=1$.
\end{proof}

The case $s=3$ (that is $n=2$), is 
 is essentially
due to \cite{act}. The proof given there can be modified to fit our case which gives the following proposition.

\begin{prop}
Let $\mcB=\mcQ+C_1+C_2$. Then $\Delta_{\mcB}\not=1$ if and only if there exists a conic passing through the eight points $\mcQ\cap(\mcC_1\cup\mcC_2)=\{P_{11},\ldots,P_{14},P_{21},\ldots,P_{24}\}$.
\end{prop}

\section{Examples}

We refer to \cite{act} for the definition of Zariski pairs and $N$-plets.

\subsection{Example 1: $k$-Namba-Tsuchihashi arrangements ($k \ge 3$).}

In \cite{namba-tsuchi}, Namba and Tsuchihashi constructed an example of a Zariski pair of configurations of degree 8, where each configuration consists of  4 conics. The first and the third author generalized this example and constructed examples of  Zariski $N$-plets consisting of any number of conics greater than or equal to four, which they named $k$-Namba-Tsuchihashi arrangements \cite{bannai-tokunaga}. In the following we give a brief summary of the construction and the proof.

Let $\mathcal{Q}\subset\PP^2$ be a plane quartic curve consisting of two conics meeting transversely and let $P\in \mathcal{Q}$. Consider the elliptic surface $S_P$ as constructed in Section \ref{admissible}. It is immediate that $\MW(S_P)\cong\left(A_1^\ast\right)^{\oplus3}\oplus\ZZ/2\ZZ$, $\MW(S_P)^0\cong\left(A_1\right)^{\oplus3}$. Let $s_1, s_2, s_3$ be generators for the free part of $\MW(S_P)$ and let $f_P:S_{P}\rightarrow\PP^2$ be the composition of the maps described above.

In \cite{bannai-tokunaga}, the following statements were proved:

\begin{lem}\label{llem}
There exists three families of plane conics $\mathcal{F}_1, \mathcal{F}_2, \mathcal{F}_3$  such that 
\begin{enumerate}
\item Each member of $\mathcal{F}_i$ is a smooth conic, tangent to $\mathcal{Q}$ at four distinct smooth points.
\item For  $\mathcal{C}\in\mathcal{F}_i$, the proper transform $f_P^\ast(\mathcal{C})$  splits into two irreducible components $\mathcal{C}^+, \mathcal{C}^-$such that   $\tilde{\psi}_P(\mathcal{C}^\pm)=[2]s_i$
\end{enumerate}
\end{lem} 
Note that this lemma is also a consequence of Lemma \ref{subscribed} applied to the sections $[2]s_i$.
 In addition, under the notation Lemma \ref{llem}, 

\begin{lem}\label{strong-kNT}
Let $k$ be an integer $\geq 3$ and $(k_1, k_2, k_3)$ be a triplet of non-negative integers such that $k=k_1+k_2+k_3$. Then there exist conics $\mathcal{C}_1,\ldots,\mathcal{C}_k$ such that 
\begin{enumerate}
\item $\mathcal{C}_1,\ldots, \mathcal{C}_{k_1}\in\mathcal{F}_1$, $\mathcal{C}_{k_1+1},\ldots, \mathcal{C}_{k_1+k_2}\in\mathcal{F}_2$, $\mathcal{C}_{k_1+k_2+1},\ldots, \mathcal{C}_k\in\mathcal{F}_3$,
\item for any $i\not=j$, $\mathcal{C}_i$ and $\mathcal{C}_j$ meet transversely,     
\item and for any distinct $i, j, k$, $\mathcal{C}_i\cap\mathcal{C}_j\cap\mathcal{C}_k=\emptyset$ and $\mathcal{C}_i\cap\mathcal{C}_j\cap\mathcal{Q}=\emptyset$.
\end{enumerate}
\end{lem}

This Lemma allows us to construct conic configurations $\mathcal{B}^{(k_1,k_2,k_3)}=\mcQ+\mcC_1+\cdots+\mcC_k$ having the same combinatorics with prescribed conditions concerning the Abel-Jacobi map. 

\begin{lem}
Under the same notation as the above lemma,
$\Cov_b(\PP^2,2\mathcal{Q}+p\mathcal{C}_i+p\mathcal{C}_j,D_{2p})\not=\emptyset$ if and only if $\mathcal{C}_i$ and $\mathcal{C}_j$ are members of the same family.
\end{lem}

By comparing  ${\bf Cov}(\PP^2,\mathcal{B}^{(k_1,k_2,k_3)},D_{2p})$   the following theorem was proved: 

\begin{thm}\label{kNT-main}
Let $k\geq 3$ and $(k_1, k_2, k_3)$, $(k_1^\prime,k_2^\prime, k_3^\prime)$ be triples  of positive integers satisfying $k=k_1+k_2+k_3=k_1^\prime+k_2^\prime+k_3^\prime$,  $k_1\geq k_2\geq k_3$ and  $k_1^\prime\geq k_2^\prime\geq k_3^\prime$. If $(k_1,k_2,k_3)\not=(k_1^\prime, k_2^\prime, k_3^\prime)$, then the two configurations corresponding to the triples as in Lemma \ref{strong-kNT} are a Zariski-pair.
\end{thm}

The key point in proving this theorem is Corollary \ref{key-cor1}, namely to  show the non-existence of a inclusion preserving bijection $\eta:{\rm Sub}(\mcB^{(k_1,k_2,k_3)})\rightarrow{\rm Sub}(\mcB^{(k^\prime_1,k^\prime_2,k^\prime_3)})$ induced by a homeomorphism that is compatible with ${\bf Cov}(\PP^2,\bullet,G)$. The obstruction for the existence of $\eta$ is the differences in sub configurations of the form $\mathcal{Q}+\mathcal{C}_i+\mathcal{C}_j$. For details see Proposition 3.4 in \cite{bannai-tokunaga}. 
Now we will give an alternative proof of this theorem using Alexander polynomials. 

\begin{lem}[\cite{act}, Proposition 4.18] 
Let $\mathcal{Q}$ be as before and $\mathcal{C}, \mathcal{C}^\prime$ be members of $\cup\mathcal{D}_i$. Then there exists a conic passing through the 8 points of tangency of $\mathcal{Q}$ and $\mathcal{C}+\mathcal{C}^\prime$ if and only if $\mathcal{C}, \mathcal{C}^\prime$ are members of the same family.
\end{lem}

\begin{cor}
The reduced Alexander polynomial $\tilde{\Delta}_{\mathcal{Q}+\mathcal{C}+\mathcal{C}^\prime}(t)$ is non-trivial if and only if $\mathcal{C}, \mathcal{C}^\prime$ are members of the same family.
\end{cor}

Now by the same argument as above, instead of using ${\bf Cov}(\PP^2,\bullet,D_{2p})$ but by considering ${\bAlex}_{\mcB}$, we see that $\eta$ satisfying ${\bAlex}_{{\mcB^{(k_1,k_2,k_3)}}}={\bAlex}_{{\mcB^{(k_1^\prime,k_2^\prime,k_3^\prime)}}}\circ\eta$ cannot exist because of  the differences of  ${\bAlex}_{{\mcB^{(k_1,k_2,k_3)}}}$ on the subconfigurations of the form $\mcQ+\mcC_i+\mcC_j$. Hence by Corollary \ref{key-cor2}, we have our result.
\bigskip

\subsection{Example 2: A quartic with two nodes and three conics}
 
Let $S$ be the rational elliptic surface corresponding to the Weierstrass equation
\begin{align*}
 y^2={x}^{3}+ \left( 271350-98\,t \right) {x}^{2}+t \left( t-5825 \right) 
 \left( t-2025 \right) x+36\,{t}^{2} \left( t-2025 \right) ^{2}.
 \end{align*}
 This corresponds to considering the construction as in Section \ref{admissible} for the case $\mathcal{Q}: {x}^{3}+ \left( 271350-98\,t \right) {x}^{2}+t \left( t-5825 \right) 
 \left( t-2025 \right) x+36\,{t}^{2} \left( t-2025 \right) ^{2}=0$ and $P=[0:1:0]$.
 $S$ has two singular fibers of type $\I_2$ and one singular fiber of type $\III$. The Mordell-Weil lattice of $S$ is  $\MW(S)\cong D_4^\ast\oplus A_1^\ast$ and
 the narrow Mordell-Weil lattice is  $\MW(S)^0\cong D_4\oplus A_1$. Considering $S$ as an elliptic curve over $\mathbb{C}(t)$, the following rational points $s_1,\cdots, s_5$  form a basis for $\MW(S)$ (\cite{shioda-usui}). 
 \begin{align*}
 s_1=&\left(\frac{1}{25}t(t-2025), \frac{1}{125}(t^2-3950t+3898125)t \right)\\
 s_2=&\left( {\frac {1}{100}}{t}^{2}+{\frac {75}{2}}
t-{\frac {669375}{4}}, -{\frac {1}{1000}}{t}^{3}-{\frac {229}{40}}{
t}^{2}+{\frac {368625}{8}}t-{\frac {431746875}{8}}
 \right)\\
 s_3=&\left( -32t, 2t \left( t-3465 \right) \right)\\
 s_4=&\left( -35t+70875,-\left( t+20475 \right) 
 \left( t-2025 \right)\right)\\
s_5=&\left(0, 6t \left( t-2025 \right) \right)
 \end{align*}
 It can be checked that $\langle s_i,s_5\rangle=0$ ($i=1,\ldots4$), $\langle s_5, s_5\rangle=\frac{1}{2}$, and for $1\leq i,j\leq 4$, 
 \[
 [\langle s_i, s_j\rangle]=\left[ \begin {array}{cccc} 1&1& \frac{1}{2} &\frac{1}{2}\\\noalign{\medskip}1&2&1&1
\\\noalign{\medskip}\frac{1}{2}&1&1&\frac{1}{2}\\\noalign{\medskip}\frac{1}{2}&1&\frac{1}{2}&1
\end {array} \right].
 \]
 A set of  generators $t_1,\ldots,t_5$ for the narrow Mordell-Weil lattice is given by
 \begin{eqnarray*}
 &t_1=2s_1-s_2,\quad 
 t_2=-s_1+2s_2-s_3-s_4,\quad
 t_3=-s_2+2s_3,\\
 &t_4=-s_2+2s_4,\quad
 t_5=2s_5
 \end{eqnarray*}
 with $\langle t_i,t_5\rangle=0$ ($i=1,\ldots4$), $\langle t_5, t_5\rangle=2$ and for $1\leq i,j\leq 4$,
 \[
 [\langle t_i,t_j\rangle]=\left[ \begin {array}{cccc} 2&-1&0&0\\\noalign{\medskip}-1&2&-1&-1
\\\noalign{\medskip}0&-1&2&0\\\noalign{\medskip}0&-1&0&2\end {array}
 \right]. 
 \]
We will  consider the families of conics corresponding to the roots $t_1,t_2, t_3, t_4, t_1+t_2$. Let $\mathcal{F}_{\bullet}$ be the family corresponding to $\bullet=t_1,t_2, t_3, t_4$ or $ t_1+t_2$. By using the explicit method of calculating the equations of the conics given in \cite{bannai-tokunaga}, we obtain the following parametrized equations for the families where $a_1,\ldots, a_5$ are complex parameters.
\begin{itemize}
\item $\mathcal{F}_{t_1}$
\begin{align*}
&\left( {\frac {407}{32}}\,a_{{1}}-{\frac {454825}{256}}-{\frac {1}{
400}}\,{a_{{1}}}^{2} \right) {t}^{2}+ \left( {\frac {1165}{8}}+\frac{1}{10}\,a_{{1}} \right) tx\\
&+ \left( {\frac {914968125}{128}}-{\frac {624675}{16}
}\,a_{{1}}+{\frac {381}{8}}\,{a_{{1}}}^{2} \right) tz-{x}^{2}+ \left( 
-{\frac {4730625}{16}}+{a_{{1}}}^{2} \right) xz\\
&+ \left( {\frac {
846129375}{32}}\,a_{{1}}-{\frac {389025}{16}}\,{a_{{1}}}^{2}-{\frac {
1840331390625}{256}} \right) {z}^{2}=0
\end{align*}
\item $\mathcal{F}_{t_2}$
\begin{align*}
&\left( -{\frac {74102665}{4096}}-{\frac {12611}{768}}\,a_{{2}}-{
\frac {1}{576}}\,{a_{{2}}}^{2} \right) {t}^{2}+ \left( {\frac {5821}{
32}}-\frac{1}{12}\,a_{{2}} \right) tx\\
&+ \left( {\frac {29449125}{128}}\,a_{{2}}
+{\frac {2685}{32}}\,{a_{{2}}}^{2}+{\frac {265613914125}{2048}}
 \right) tz-{x}^{2}+ \left( -{\frac {40259025}{64}}+{a_{{2}}}^{2}
 \right) xz\\&+ \left( -{\frac {145253705625}{256}}\,a_{{2}}-{\frac {
22892625}{64}}\,{a_{{2}}}^{2}-{\frac {921634762190625}{4096}} \right) 
{z}^{2}=0
\end{align*}
\item $\mathcal{F}_{t_3}$
\begin{align*}
&\left( -{\frac {289}{12}}\,a_{{3}}-{\frac {62785}{16}}-\frac{1}{36}\,{a_{{3}}
}^{2} \right) {t}^{2}+ \left( {\frac {361}{2}}+\frac{1}{3}\,a_{{3}} \right) tx\\
&+ \left( {\frac {106713693}{8}}+{\frac {139527}{2}}\,a_{{3}}+{\frac {
165}{2}}\,{a_{{3}}}^{2} \right) tz-{x}^{2}+ \left( -{\frac {1225449}{4
}}+{a_{{3}}}^{2} \right) xz\\
&+ \left( -{\frac {155034243}{4}}\,a_{{3}}-{
\frac {171622907001}{16}}-{\frac {140049}{4}}\,{a_{{3}}}^{2} \right) {
z}^{2}=0
\end{align*}
\item $\mathcal{F}_{t_4}$
\begin{align*}
&\left( {\frac {227}{96}}\,a_{{4}}+{\frac {1275575}{256}}-{\frac {1}{
144}}\,{a_{{4}}}^{2} \right) {t}^{2}+ \left( {\frac {925}{8}}+\frac{1}{6}\,a_{{4}} \right) tx\\&+ \left( -{\frac {1317211875}{128}}+{\frac {182475}{16}
}\,a_{{4}}+{\frac {141}{8}}\,{a_{{4}}}^{2} \right) tz-{x}^{2}+ \left( 
-{\frac {4100625}{16}}+{a_{{4}}}^{2} \right) xz\\&+ \left( -{\frac {
487974375}{32}}\,a_{{4}}+{\frac {240975}{16}}\,{a_{{4}}}^{2}+{\frac {
988148109375}{256}} \right) {z}^{2}=0
\end{align*}
\item $\mathcal{F}_{t_1+t_2}$
\begin{align*}
& \left( -{\frac {6107545}{256}}+{\frac {2501}{32}}\,a_{{5}}-\frac{1}{16}\,{a_{{5}}}^{2} \right) {t}^{2}+ \left( {\frac {2629}{8}}-\frac{1}{2}\,a_{{5}}
 \right) tx\\&+ \left( {\frac {10772038125}{128}}-{\frac {4466025}{16}}\,
a_{{5}}+{\frac {1845}{8}}\,{a_{{5}}}^{2} \right) tz-{x}^{2}+ \left( {a_{{5}}}^{2}-{\frac {6579225}{16}} \right) xz\\&+ \left( {\frac {
5739508125}{32}}\,a_{{5}}-{\frac {2237625}{16}}\,{a_{{5}}}^{2}-{\frac 
{14721838340625}{256}} \right) {z}^{2}=0
\end{align*}
\end{itemize}
By direct computations, using the computer algebra software MAPLE, we can prove the following:

\begin{lem}\label{existanceofconfig}
Let $n$ be any natural number. Then it is possible to choose smooth conics $\mathcal{C}_1,\ldots,\mathcal{C}_n$ so that
\begin{enumerate}
\item $\mathcal{C}_i$ $(i=1,\ldots,n)$ can be chosen from any one of the families $\mathcal{F}_\bullet$ given above,
\item $\mathcal{C}_i$ is tangent to  $\mathcal{Q}$  at $4$ smooth points of $\mathcal{Q}$,
\item $\mathcal{C}_i$ and $\mathcal{C}_j$ intersect transversally at $4$ points,
\item no three of  $\mathcal{Q}, \mathcal{C}_1,\ldots, \mathcal{C}_n$ intersect at a single point.
\end{enumerate}
In particular, for fixed $n$, we can construct infinitely many configurations $\mathcal{Q}+\mathcal{C}_1+\cdots+\mathcal{C}_n$ that have the same combinatorics, with $\mathcal{C}_i$ having prescribed conditions on $\tilde{\psi}(\mcC_i^{\pm})$. 
\end{lem}

Consider the configurations $\mathcal{B}^1,\ldots,\mathcal{B}^4$ where the components are chosen as follows:
\begin{eqnarray*}
\mathcal{B}^1=\mcB^1_1+\mcB^1_2+\mcB^1_3+\mcB^1_4=\mathcal{Q}+\mathcal{C}^1_1+\mathcal{C}^1_2+\mathcal{C}^1_3 & \mathcal{C}^1_2,\mathcal{C}^1_2, \mathcal{C}^1_3\in \mathcal{F}_{t_1}\\
\mathcal{B}^2=\mcB^2_1+\mcB^2_2+\mcB^2_3+\mcB^2_4=\mathcal{Q}+\mathcal{C}^2_1+\mathcal{C}^2_2+\mathcal{C}^2_3 & \mathcal{C}^2_1, \mathcal{C}^2_2\in \mathcal{F}_{t_1}, \mathcal{C}^2_3\in \mathcal{F}_{t_2}\\
\mathcal{B}^3=\mcB^3_1+\mcB^3_2+\mcB^3_3+\mcB^3_4=\mathcal{Q}+\mathcal{C}^3_1+\mathcal{C}^3_2+\mathcal{C}^3_3 & \mathcal{C}^3_1\in\mathcal{F}_{t_1}, \mathcal{C}^3_2\in\mathcal{F}_{t_2}, \mathcal{C}^3_3\in\mathcal{\mathcal{F}}_{t_3}\\
\mathcal{B}^4=\mcB^4_1+\mcB^4_2+\mcB^4_3+\mcB^4_4=\mathcal{Q}+\mathcal{C}^4_1+\mathcal{C}^4_2+\mathcal{C}^4_3 & \mathcal{C}^4_1\in\mathcal{F}_{t_1}, \mathcal{C}^4_2\in\mathcal{F}_{t_2}, \mathcal{C}^4_3\in\mathcal{F}_{t_1+t_2}
\end{eqnarray*}
We will assume that the conics $\mcC^i_j$ in each $\mathcal{B}^i$  satisfies the conditions in Corollary \ref{existanceofconfig}, so $\mathcal{B}^1,\ldots, \mathcal{B}^4$ will have the same combinatorics. In this situation, by Lemma \ref{criterion},  
$\Cov_b(2\mathcal{Q}+p\mathcal{C}^k_i+p\mathcal{C}^k_j,D_{2p})\not=\emptyset$ if and only if $\mathcal{C}^k_i$ and $\mathcal{C}^k_j$ are members of the same family. This is enough to distinguish between $\mathcal{B}^i$ and $\mathcal{B}^{j}$ except for the case where $(i,j)=(3,4)$. To distinguish $\mathcal{B}^3$ and $\mathcal{B}^4$, we further need to consider $\Cov_b(\mathcal{B}^i(2,p,p,p),D_{2p})$. Again, by Lemma \ref{criterion}, since $t_1, t_2, t_3$ are linearly independant in $MW(S_x)$, $\Cov_b(\mathcal{B}^3(2,p,p,p),D_{2p})=\emptyset$, where as, since $t_1, t_2, t_1+t_2$ are linearly dependent, $\Cov_b(\mathcal{B}^4(2,p,p,p),D_{2p})\not=\emptyset$. Summing up, we have the following table which shows the existence/non-existence of $D_{2p}$-covers with ramification data as specified at the heads of the columns. 
\begin{center}
\begin{tabular}{c||c|c|c||c}
& $\mathcal{B}^i(2,p,p,1)$ &  $\mathcal{B}^i(2,p,1,p)$ &  $\mathcal{B}^i(2,1,p,p)$&  $\mathcal{B}^i(2,p,p,p)$\\
\hline
$\mathcal{B}^1$ & $\exists$ & $\exists$ & $\exists$ & $\exists$\\
$\mathcal{B}^2$ & $\exists$ & $\emptyset$ & $\emptyset$ & $\emptyset$\\
$\mathcal{B}^3$ &$\emptyset$ & $\emptyset$ & $\emptyset$ & $\emptyset$\\
$\mathcal{B}^4$ &$\emptyset$ & $\emptyset$ &$\emptyset$ & $\exists$ 
\end{tabular}
\end{center}
We note that a homeomorphism $h:(\PP^2,\mathcal{B}^i)\rightarrow(\PP^2,\mathcal{B}^j)$ must satisfy $h(\mcQ)=\mcQ$. From the table above, it is clear that  ${\bf Cov}(\PP^2,\mathcal{B}^i,D_{2p})\not\underset{\eta}{\approx}{\bf Cov}(\PP^2,\mathcal{B}^j,D_{2p})$ for any $\eta$ induced by a homeomorphism.
Hence by Corollary \ref{key-cor1} we have the following proposition.
\begin{prop}
The configurations $(\PP^2,\mathcal{B}^1),\ldots, (\PP^2,\mathcal{B}^4)$ form a Zarisi 4-plet.
\end{prop}
We emphasize again that the four arrangements are not  distinguishable if we look only at  $\Cov_b(\PP^2,\mathcal{B}^i(2,p,p,p),D_{2p})$, an invariant for the whole configuration. We were only able to distinguish them by considering the subconfigurations of $\mathcal{B}^i$ and  ${\bf Cov}(\PP^2,\mathcal{B}^i,D_{2p})$.


Now we will consider $\bAlex(\mathcal{B}^i)$. By direct calculations we can prove the following:
\begin{lem}
It is possible to choose conics $\mathcal{C}^i_j$, $j=1,2,3$ for $\mathcal{B}^i$ so that there exists a conic passing through the eight points of tangency of $\mathcal{Q}$ and $\mathcal{C}^i_k+\mathcal{C}^i_l$ if and only if $\mathcal{C}^i_k$ and $\mathcal{C}^i_l$ are from the same family.
\end{lem}

We note that the statement is true for all $\mathcal{C}^i_j$  but we only need this special case in the following argument.

\begin{cor}
Let $\mathcal{C}^i_j$ satisfy the conditions of the above lemma. Then
$\tilde{\Delta}_{\mathcal{Q}+\mathcal{C}^i_k+\mathcal{C}^i_l}(t)=t^2+1$  if and only if $\mathcal{C}^i_k$ and $\mathcal{C}^i_l$ are of the same family. Otherwise, $\tilde{\Delta}_{\mathcal{Q}+\mathcal{C}^i_k+\mathcal{C}^i_l}(t)=1$.
\end{cor}

The following table gives $\tilde{\Delta}_{\mcB_I^i}(t)$.

\begin{center}
\begin{tabular}{c||c|c|c||c}
& $\tilde{\Delta}_{\mathcal{Q}+\mathcal{C}_{i,1}+\mathcal{C}_{i,2}}(t)$ &  $\tilde{\Delta}_{\mathcal{Q}+\mathcal{C}_{i,1}+\mathcal{C}_{i,3}}(t)$ &  $\tilde{\Delta}_{\mathcal{Q}+\mathcal{C}_{i,2}+\mathcal{C}_{i,3}}(t)$ & $\tilde{\Delta}_{\mathcal{Q}+\mathcal{C}_{i,1}+\mathcal{C}_{i,2}+\mathcal{C}_{i,3}}(t)$\\
\hline
$\mathcal{B}^1$ & $t^2+1$ & $t^2+1$ & $t^2+1$ & $1$\\
$\mathcal{B}^2$ & $t^2+1$ & $1$ & $1$ & $1$\\
$\mathcal{B}^3$ &$1$ & $1$ & $1$ & $1$\\
$\mathcal{B}^4$ &$1$ & $1$ &$1$ & $1$ 
\end{tabular}
\end{center}

In this case, we can distinguish $\mcB^i$ and $\mcB^j$ by Corollary \ref{key-cor2} for $(i,j)\not=(3,4)$ as $\bAlex_{\mcB^i}\not=\bAlex_{\mcB^j}\circ\eta$ for any $\eta$ induced by a homeomorphism for $(i,j)\not=(3,4)$, but we cannot distinguish $\mcB^3$ and $\mcB^4$.
As a conclusion, we can say that ${\bf Cov}(\PP^2,\mcB,G)$ is a finer invariant compared to  $\bAlex_{\mcB}$.

As a final remark, we note that $\mathcal{B}^3$ and $\mathcal{B}^4$ can be distinguished by  an elementary  geometric condition. There exists a cubic passing through the 12 points of tangency of $\mathcal{Q}$ and $\mathcal{C}^3_1+\mathcal{C}^3_2+\mathcal{C}^3_3$ where as no such cubic exists for $\mathcal{B}^4$.

\noindent Shinzo BANNAI and  Hiro-o TOKUNAGA\\
Department of Mathematics and Information Sciences\\
Graduate School of Science and Engineering,\\
Tokyo Metropolitan University\\
1-1 Minami-Ohsawa, Hachiohji 192-0397 JAPAN \\
{\tt shinzo.bannai@gmail.com}, {\tt tokunaga@tmu.ac.jp}

\vspace{0.5cm}
      
\noindent Masayuki KAWASHIMA\\
      Department of Mathematics,\\
         Tokyo University of Science,\\
         1-3 Kagurazaka, Shinjuku-ku, 
         Tokyo 162-8601 JAPAN\\
{\tt kawashima@ma.kagu.tus.ac.jp}
 \end{document}